\documentclass[12pt]{amsart}
\usepackage{amssymb}
\usepackage{amscd}
\usepackage{mathrsfs}
\usepackage{a4wide}
\usepackage{color}
\usepackage{xypic}

\def\mathcal#1{\mathscr{#1}}

\def\con{\subseteq}
\def\from{\colon}

\newcommand\reals{\mathbb R}

\newcommand\CDH{\ensuremath{\mathsf{CDH}}}

\newcommand\cont{\mathfrak c}

\newsymbol\boog 1061        

\newcommand\CH{\ensuremath{\mathsf{CH}}}

\newcommand{\darrow}{\mathbb{A}}
\newcommand{\pair}[1]{\langle #1 \rangle}
\newcommand{\I}{\mathcal{I}}

\newcommand{\cantorset}{{2}\sp\omega}
\newcommand{\cantortree}{{2}\sp{<\omega}}
\newcommand{\Z}{\mathbb{Z}}

\newtheorem{theorem}{Theorem}[section]
\newtheorem{corollary}[theorem]{Corollary}
\newtheorem{proposition}[theorem]{Proposition}
\newtheorem{lemma}[theorem]{Lemma}

\theoremstyle{remark}

\newtheorem{question}[theorem]{Question}

\theoremstyle{definition}
\newtheorem{definition}[theorem]{Definition}

\def\leukfrac#1/#2{\leavevmode
               \kern.1em
                \raise.9ex\hbox{\the\scriptfont0 ${}_#1$}
                \hskip -1pt\kern-.1em
                /\kern-.15em\lower.10ex\hbox{\the\scriptfont0 ${}_#2$}}

\makeatletter
\def\diam{\mathop{\operator@font diam}\nolimits}
\makeatother

\def\hrusak{Hru{\v{s}}{\'a}k}
\def\zamora{Zamora Avil{\'e}s}
\def\pelz{Pe{\l}czy\'{n}ski}

\begin{document}
\title[$\lambda$-sets]{Countable dense homogeneity and $\lambda$-sets}

\author[Hern\'andez-Guti\'errez]{Rodrigo Hern\'andez-Guti\'errez${}^1$}

\address{${}^1$ Department of Mathematics and Statistics, 
York University, Toronto, ON M3J 1P3, Canada} 

\email{rodhdz@yorku.ca}

\thanks{The first-listed author would like to thank the Centro de Ciencias Matem\'aticas at Morelia for
its support during his PhD studies there}

\author[Hru\v{s}\'ak]{Michael Hru\v{s}\'ak${}^2$}

\address{${}^2$Centro de Ciencias Matem\'aticas\\
UNAM\\
A.P. 61-3, Xangari, Morelia, Michoac\'an\\
58089, M\'exico}

\email{michael@matmor.unam.mx}

\thanks{{The second-listed author was supported by a PAPIIT grant IN 102311 and CONACyT grant 177758.}}

\author[van Mill]{Jan van Mill${}^3$}
\address{${}^3$Faculty of Sciences,
VU University Amsterdam, De
Boelelaan 1081A, 1081 HV Amsterdam, The Netherlands}

\address{Faculty of Electrical Engineering, Mathematics and Computer Science, TU
Delft, Postbus 5031, 2600 GA Delft, The Netherlands}

\address{Department of Mathematical Sciences, University of South Africa,
P. O. Box 392,
0003 Unisa,
South Africa}
\email{j.van.mill@vu.nl}

\thanks{The third-listed author is pleased to thank the Centro de Ciencias Matem\'aticas at Morelia for generous hospitality and support.}

\date{\today}



\begin{abstract} We show that all sufficiently nice $\lambda$-sets are countable dense homogeneous (\CDH).
From this fact we conclude that for every uncountable cardinal $\kappa \le \mathfrak{b}$ there is a countable dense homogeneous metric space of size $\kappa$.
Moreover, the existence of a meager in itself countable dense homogeneous metric space of size $\kappa$ is equivalent to the existence of a
$\lambda$-set of size $\kappa$. On the other hand, it is consistent with the continuum arbitrarily large that every  {\CDH} metric space has size either
$\omega_1$ or size $\mathfrak c$. An example of a Baire {\CDH} metric space which is not completely metrizable is presented. Finally, answering a question of Arhangel'skii
and van Mill we show that that there is a compact non-metrizable {\CDH} space in ZFC.
\end{abstract}

\maketitle

\section{Introduction}\label{introduction}

A separable topological space $X$ is \emph{countable dense homogeneous} (\CDH) if, given any two countable dense subsets $D$ and $E$ of $X$, there is
a homeomorphism $f\from X\to X$ such that $f[D]=E$. This is a classical notion \cite{bennett} that can be traced back to the works of Cantor, Brouwer, Fr\'echet, and others (see \cite{arh-vm-homogeneity-review} and \cite{openprobCDH}).

Examples of
\CDH-spaces are the Euclidean spaces, the Hilbert cube and the Cantor set. In fact, every strongly locally homogeneous Polish space is \CDH, as was shown
by Bessaga and \pelz~\cite{bp:estimated}. Recall that a space $X$ is \emph{strongly locally homogeneous} if it has a basis $\mathcal B$ of open sets such that
for every $U\in \mathcal B$ and every $x,y\in U$ there is a autohomeomorphism $h$ of $X$ such that
$h(x)=y$ and $h$ restricts to the identity on $X\setminus U$. All these results are based on the completeness of the spaces involved.

Countable dense homogeneity had for a long time been studied mostly as a geometrical notion (\cite{bennett}, \cite{fitz-CDH-locconn}) and only relatively
recently set-theoretic methods entered the picture. At first the use of set theory was restricted to
constructions of non-completely metrizable {\CDH}-spaces assuming special set theoretic axioms
such as CH (\cite{FitzpatrickZhou92}, \cite{saltsman1}, \cite{saltsman2}) or versions of Martin's Axiom (\cite{CDH_MA}). In fact, it was only recently
that a ZFC example of a non-complete {\CDH}-space was given by Farah, \hrusak\ and Mart\'{\i}nez Ranero~\cite{FarahHrusakMartinez05}.
They proved that there exists a \CDH-subset of $\reals$ of size $\aleph_1$.
 The proof given in \cite{FarahHrusakMartinez05} uses forcing combined with an absoluteness argument involving infinitary logic
(the so-called Keisler Compactness Theorem), while giving no hint as to how to
produce an ``honest" ZFC proof. One of the main purposes of this paper is to provide just that.

Note that every countable \CDH-space is discrete,
hence $\aleph_1$ is the first cardinal where anything of \CDH-interest can happen.
Since $\reals$ is \CDH\ and has size $\cont$, it is an interesting problem as to what can happen
for cardinals greater than $\aleph_1$ but below $\cont$. In this paper we will address this problem too.

It is easy to see that every \CDH-space is a disjoint sum of a Baire {\CDH}-space and a meager in itself
{\CDH}-space. The example given in \cite{FarahHrusakMartinez05} is meager in itself. It is a result of Fitzpatrick and Zhou \cite{FitzpatrickZhou92} that
every meager in itself {\CDH}-space is, in fact, a $\lambda$-set. Recall that a separable metric space $X$ is a
\emph{$\lambda$-set} if every countable subset of $X$ is a relative $G_\delta$-set.
Here we show that for every $\lambda$-set there is another $\lambda$-set of the same
cardinality which is {\CDH}. This not only provides an honest construction of a separable metric {\CDH}-space which is not completely metrizable 
but also shows that there can be many distinct cardinalities
of {\CDH} metric spaces, as there is a meager in itself {\CDH} metric space of size $\kappa$ if and only if
there is $\lambda$-set of the same cardinality. In particular, there is a {\CDH} metric space of size
$\kappa$ for every $\kappa\leq \mathfrak b$.
On the other hand, we show that it is also consistent
with the continuum arbitrarily large that every meager in itself {\CDH} metric space has size
$\omega_1$ while every Baire {\CDH} metric space has size $\mathfrak c$.

We present a ZFC construction of a Baire {\CDH} metric space which is not completely metrizable, thus providing a very
different absolute example of a {\CDH} metric space which is not completely metrizable. The space is the
complement (in a completion) of a carefully chosen $\lambda$-set (in fact, a $\lambda '$-set).

Finally,  we deal with the existence of non-metrizable {\CDH} compact spaces. We again use  $\lambda '$-sets to define linearly ordered non-metrizable  {\CDH} 
compacta (variants of the \emph{double arrow} space).
This provides the first ZFC examples of non-metrizable {\CDH}-compacta and answers a question of Arhangel'skii and van Mill \cite{arh-vm-homogeneity-review}.

\section{$\lambda$-sets}

A \emph{$\lambda$-set} is a subset $X$ of $2^\omega$ such that every countable subset of $X$ is relative $G_\delta$. This notion is due to Kuratowski
\cite{kuratowski-lambda}. The existence of uncountable $\lambda$-sets was proved by Lusin \cite{lusin-lambda} and was later improved by Rothberger \cite{rothberger-lambda}:
\emph{There exist $\lambda$-sets of size $\kappa$ for every $\kappa \le \mathfrak{b}$}. A subset $X$ of $2^\omega$ is a \emph{$\lambda'$-set} if for every countable
subset of $Y\subseteq 2^\omega$,  $Y$ is relative $G_\delta$ in $X\cup Y$. Rothberger \cite{rothberger-lambda} has also shown that there is a $\lambda$-set of size
$\mathfrak{b}$, which is not a $\lambda'$-set, while every set of size less than $\mathfrak b$ is a $\lambda$-set, hence also a $\lambda'$-set.
Sierpi\'nski \cite{sierpinski-lambda} noted that a union of countably many $\lambda'$-sets is a $\lambda'$-set.

Every $\lambda$-set is meager in itself. Topologically, $\lambda$-sets are characterized as follows: they are precisely the zero-dimensional spaces
having the property that all countable sets are  $G_\delta$. We will call such spaces $\lambda$-sets as well.

\begin{lemma}\label{prodlamb}
Let $X$ and $Y$ be $\lambda$-sets. Then so is $X\times Y$.
\end{lemma}

\begin{proof}
Let $A\con X\times Y$ be countable. Let $\pi\from X\times Y\to X$ denote the projection.
Then $X\setminus \pi[A]$ is $F_\sigma$ by assumption, hence so is $(X\setminus \pi[A]) \times Y$.
Since for every $p\in \pi[A]$ we have that $(\{p\}\times Y)\setminus A$ is $F_\sigma$ in $\{p\}\times Y$
and hence in $X\times Y$, we are done.
\end{proof}

\begin{lemma}
If a space $X$ has a countable closed cover by $\lambda$-sets, then $X$ is a $\lambda$-set.
\end{lemma}

\begin{proof}
First observe that $X$ is zero-dimensional by the Countable Closed Sum Theorem (\cite[1.3.1]{engelking:dim:nieuw}).
Since subspaces of $\lambda$-sets are $\lambda$-sets, we may assume by zero-dimensionality that $X$ is covered by a disjoint
countable family $\mathcal{L}$ of relatively closed $\lambda$-sets. Now, if $A\con X$ is countable, then for every $L\in \mathcal{L}$, $L\setminus A$
is $F_\sigma$ in $L$ and hence in $X$. Hence $X\setminus A$ is $F_\sigma$.
\end{proof}

Meager in themeselves \CDH-subspaces of $\reals$ are $\lambda$-sets, as can be seen from the following observation. If $X$ is first category,
then it contains a countable dense set $D$ which is $G_\delta$ in $X$. This simple but useful fact, was proved by Fitzpatrick and Zhou~\cite{FitzpatrickZhou92}
(and was put to good use in \hrusak\ and \zamora~\cite{HrusakZamora05}). Every countable subset $A$ of $X$ can be extended to a countable dense subset of $X$
which consequently must be $G_\delta$ since $D$ is $G_\delta$, and $X$ is \CDH. Hence $A$ is $G_\delta$.

\smallskip

For 
$x,y \in  2^\omega$ we say that
$$x\sim y\text{ if }\exists m,n \in\omega \  \forall k\in\omega \ x(m+k)=y(n+k).$$
The relation $\sim$, known as the \emph{tail equivalence},  is a Borel equivalence relation on $2^\omega$ with countable and dense equivalence classes.
Given a set $X\subseteq 2^\omega$ we define its \emph{saturation} $X^*=\{y\in 2^\omega: \exists x\in X \  y\sim x\}$.
We will call a set $X\subseteq 2^\omega$ \emph{saturated} if it is saturated with respect to $\sim$, i.e. if $X=X^*$.

Given $s\in 2^{<\omega}$, we denote by $[s]=\{x\in 2^\omega: s\subseteq x\}$ the basic clopen set (the \emph{cone}) determined by $s$.
Given $s,t \in 2^{<\omega}$ we let $h_{s,t}: [s]\rightarrow [t]$ be defined by $h_{s,t}(s^{\smallfrown}x)= t^{\smallfrown}x$, for every $x\in 2^\omega$.
Then $h_{s,t}$ is a natural (even monotone with respect to to the lexicographic order on  $2^\omega$) homeomorphism between the clopen sets $[s]$ and $[t]$.
The crucial property we use is that $h_{s,t}$ \emph{respects} the equivalence relation $\sim$: $h_{s,t}(x)\sim x$ for all $x\in[s]$.

Next we see that there are many saturated $\lambda$-sets and $\lambda'$-sets.

\begin{lemma}\label{lambda}
${}$
\begin{enumerate}
\item If $X$ is a $\lambda'$-set  then $X^*$ is a $\lambda'$-set.
\item For every $\lambda$-set $X\subseteq 2^\omega$ there is a saturated $\lambda$-set $Y$ of the same cardinality.
\end{enumerate}
\end{lemma}

\begin{proof}
 The first clause follows directly from Sierpi\'nski's observation that a union of countably many $\lambda'$-sets is a $\lambda'$-set.

To prove the second clause, fix first an embedding $\varphi:2^\omega\rightarrow 2^\omega$ such that for distinct $x,y\in 2^\omega$
$\varphi(x)\not\sim\varphi(y)$. Such an embedding exists by a theorem of Silver (see \cite{silver-equiv-classes}). Then let $Y= \varphi[X]^*$. Then $Y$ is
a $\lambda$-set by the previous lemma.
\end{proof}

\section{Knaster-Reichbach covers}

Let $X$ and $Y$ be zero-dimensional spaces, let $A\con X$ be closed and nowhere dense in $X$ and let $B\con Y$ be closed and nowhere dense in $Y$. Moreover,
let $h\from A\to B$ be a homeomorphism. A triple $\langle \mathcal{U}, \mathcal{V}, \alpha\rangle$ is called a \emph{Knaster-Reichbach cover}, or
\emph{KR}-cover,   for $\langle X\setminus A, Y\setminus B, h \rangle$ if the following conditions are satisfied:
\begin{enumerate}
\item $\mathcal{U}$ is a partition of $X\setminus A$ by nonempty clopen subsets of $X$,
\item $\mathcal{V}$ is a partition of $Y\setminus B$ by nonempty clopen subsets of $Y$,
\item $\alpha\from \mathcal{U}\to \mathcal{V}$ is a bijection,
\item if for every $U\in \mathcal{U}$, $g_U\from U\to \alpha(U)$ is a bijection, then the combination mapping
$\tilde h = h \cup \bigcup_{U\in \mathcal{U}} g_U$ is continuous at all points of $A$,
and its inverse $\tilde h^{-1}$ is continuous at all points of $B$.
\end{enumerate}

KR-covers were used by Knaster and Reibach~\cite{KnasterReichbach} to prove homeomorphism extension results in the class of all zero-dimensional spaces.
The term KR-cover was first used by van Engelen~\cite{fons86} who proved their existence in a general setting.

\begin{lemma}[{\cite[Lemma 3.2.2]{fons86}}]\label{fons}
Let $X$ and $Y$ be zero-dimensional separable metrizable spaces, and let $A$ and $B$
be non-empty closed nowhere dense subspaces of $X$ and $Y$, respectively. If
$h\from A \to B$ is a homeomorphism, then there exists a KR-cover for $\langle X\setminus A,Y\setminus B,h\rangle $.
\end{lemma}

\section{{\CDH} metric spaces from $\lambda$-sets}

Perhaps the main result of this paper is the following:

\begin{theorem}\label{main-thm}
Let $X\subseteq 2^\omega$ be an uncountable saturated $\lambda$-set. Then $X$ is  a \emph{relatively} \CDH-subspace of $2^\omega$, i.e. for every
$D_0, D_1$ countable dense subsets of $X$ there is a homeomorphism $h$ of $2^\omega$ such that
\begin{enumerate}
\item $h[D_0]=D_1$, and
\item $h[X]=X$.
\end{enumerate}
\end{theorem}
\begin{proof}
Start by fixing a metric on $\cantorset$. Observe that $X$, being saturated, is dense in $2^\omega$. Since $X$ is uncountable, there exists a saturated set $E_i\subset X$ such that $D_i\subset E_i$ and $E_i\setminus D_i$ is countable dense for each $i\in2$. Let $E=E_0\cup E_1$ and $Y=X\setminus E$, notice that these sets are saturated as well.
Since $X$ is a $\lambda$-set, there is a collection 
$\{K_n:n<\omega\}$ of closed nowhere dense subsets of $\cantorset$ such that $Y=X\cap(\bigcup\{K_n:n<\omega\})$. 
Enumerate $E$ as $\{d_k: k < \omega\}$.

We will contruct the homeomorphism $h$ recursively. In step $n<\omega$, we will find a homeomorphism $h_n:\cantorset\to\cantorset$ and a pair of closed nowhere dense sets 
$G_n\sp{0}$ and $G_n\sp{1}$ such that:
\begin{itemize}
\item[(a)] For $i\in 2$, $G_n\sp{i}\cup K_n\subset G_{n+1}\sp{i}$.
\item[(b)] For $i\in 2$, $E\cap G_n\sp{i}$ is finite and contains $\{d_k:k< n\}$.
\item[(c)]  $h_{n}[G_n\sp0]=G_n\sp1$, $h_n[D_0\cap G_n\sp0]=D_1\cap G_n\sp1$, $h_n[E\cap G_n\sp0]=E\cap G_n\sp1$ and $h_n[Y]=Y$.
\item[(d)] For every $k<n$, $h_n\!\!\restriction_{G_k\sp0}=h_k$.
\end{itemize}

We will also need a 
KR-cover $\langle \I_n\sp{0},\I_n\sp{1},\alpha_n\rangle$ for $\langle \cantorset\setminus G_n\sp{0},
\cantorset\setminus G_n\sp{1},h_n\!\!\restriction_{G_n\sp{0}}\rangle$ satisfying the following:

\begin{itemize}
\item[(e)] Given $i\in2$, if $I\in\I_n\sp{i}$, then $I$ has diameter $<\frac{1}{n+1}$.
\item[(f)] For each $i\in2$, $\I_{n+1}\sp{i}$ refines $\I_n\sp{i}$; that is, every element of $\I_{n+1}\sp{i}$ is contained in an element 
of $\I_n\sp{i}$.
\item[(g)] If $I\in\I_n\sp{0}$, then $h_n[I]=\alpha_n(I)$ and $h_n[I\cap Y]=\alpha_n(I)\cap Y$.
\item[(h)] Given $m<n$, if $I\in\I_{m}\sp{0}$ and $J\in\I_{n}\sp{0}$ then $J\subset I$ if and only if $\alpha_n(J)\subset\alpha_m(I)$.
\end{itemize}

In the first step of the construction, we let $G_0\sp0=G_0\sp1=\emptyset$, $\I_0\sp0=\I_0\sp1=\{\cantorset\}$, $\alpha_0=\{\pair{\cantorset,\cantorset}\}$ and 
we set $h_0$ to be the identity function. 

So assume we are in step $n+1$ for some $n\in\omega$.

First, let $G=G_n\sp0\cup K_n\cup\{d_n\}$. We may assume that $d_n\notin G_n\sp0$ and let $I_0\in\I_n\sp{0}$ be such that $d_n\in I_0$.
Now, choose $e\in E_1\cap\alpha_n(I_0)$ in such a way that $e\in D_1$ if and only if $d_n\in D_0$; this is possible by the choice of $E_0$ and $E_1$. Let $\{I\sp{0}_0(m):m<\omega\}$ be a partition of $I_0\setminus\{d_n\}$ into infinitely many
clopen subsets. Similarly, let $\{I\sp{1}_0(m):m<\omega\}$ be a partition of $\alpha_n(I_0)\setminus\{e\}$ into infinitely many clopen subsets. 
Define $\I\sp\prime=(\I\sp{0}_n\setminus\{I_0\})\cup\{I\sp{0}_0(m):m<\omega\}$ and $\beta=\alpha_n\cup\{\pair{I\sp{0}_0(m),I\sp{1}_0(m)}:m<\omega\}$.

Let $U\in\I\sp\prime$ and $V=\beta(U)$. Every clopen subset of $\cantorset$ is a finite union of pairwise disjoint sets of the form $[s]$ 
for $s\in\cantortree$. Thus we may assume that $U=[s_0]\cup\ldots\cup[s_k]$ and $V=[t_0]\cup\ldots\cup[t_k]$ where these sets are pairwise disjoint. 
Then the function $f_U:U\to V$ defined by $f_U(x)=h_{s_j,t_j}(x)$ if $x\in[s_j]$ is a homeomorphism such that $f_U[U\cap Y]=V\cap Y$.

Then we define
$$
H(x)=\left\{
\begin{array}{ll}
h_n(x), & \textrm{ if } x\in G_n\sp0,\\
e, & \textrm{ if }x=d_n,\\
f_U(x), & \textrm{ if } x\in U\in\I\sp\prime.
\end{array}
\right.
$$

Then $H:\cantorset\to \cantorset$ is a homeomorphism. By Lemma \ref{fons}, for every $I\in\I\sp\prime$ there is a KR-cover $\pair{\I\sp{0}(I),\I\sp{1}(I),\alpha_I}$ for $\pair{I\setminus G,\beta(I)\setminus H[G],H\!\!\restriction_{G\cap I}}$. Define $\I\sp{0}=\bigcup\{\I\sp{0}(I):I\in\I\sp\prime\}$, $\I\sp{1}=\bigcup\{\I\sp{1}(I):I\in\I\sp\prime\}$ and $\alpha=\bigcup\{\alpha_I:I\in\I\sp\prime\}$. Then  $\pair{\I\sp0,\I\sp1,\alpha}$ is a KR-cover for $\pair{\cantorset\setminus G,\cantorset\setminus H[G],H\!\!\restriction_{G}}$. If necessary, this KR-cover may be refined to a KR-cover where all clopen subsets involved have diameter $<\frac{1}{n+2}$.

It is not hard to see that $G$, $\alpha$, $\I\sp0$, $\I\sp1$ and $H$ have the desired properties (a) - (h) when $i=0$.  In order to finish the induction step, we have to do another 
refinement to both the homeomorphism and the covers so that (a) to (h) hold also when $i=1$. This is entirely analogous to what we have done, so we leave the details to the reader.

Having carried out the recursive construction, observe that the 
sequences $(h_n)_n$ and $(h_n^{-1})_n$ are Cauchy (in the complete space of homeomorphisms of $2^\omega$ endowed with the topology of uniform convergence), 
hence $h = \lim_{n\to\infty} h_n$ exists and is the homeomorphism with the desired properties.
\end{proof}

In particular, if $X$ is a saturated $\lambda$-set then $X$ is \CDH, and the following corollaries easily follow.

\begin{corollary}
For every uncountable cardinal $\kappa\le \cont$, the following statements are equivalent:
\begin{enumerate}
\item There is a meager in itself {\CDH} metric space of size $\kappa$,
\item There is a $\lambda$-set of size $\kappa$.
\end{enumerate}
\end{corollary}

\begin{corollary}\label{eerstecor}
For every cardinal $\kappa$ such that $\omega_1\le\kappa\le \mathfrak{b}$ there exists a meager in itself {\CDH} metric space of size $\kappa$.
\end{corollary}

Given the latter, one has to wonder whether there is (in ZFC) a \CDH-space space of any cardinality  below $\mathfrak c$.
We will show that it is not the case.

\begin{theorem}\label{tweedestelling}
It is consistent with ZFC that the continuum is  arbitrarily large and every \CDH \  metric space has size either $\omega_1$ or $\cont$ and,
moreover,
\begin{enumerate}
\item all metric \CDH-spaces of size $\omega_1$ are $\lambda$-sets, and
\item all metric \CDH-spaces of size $\cont$ are non-meager.
\end{enumerate}
\end{theorem}

\begin{proof}
The model for this is the so called \emph{Cohen model}. That is start with a model of \CH\ and
force with $\mathbb C_\kappa$ - the partial order for adding $\kappa$-many Cohen reals, where $\kappa\geq\omega_2$. A. Miller
~\cite[Theorem 22]{Miller93} showed that all $\lambda$-sets in the Cohen model have size $\omega_1$.

\smallskip

{\bf Claim.} \emph{Let $X$ be a crowded separable metric space. If $X$ is \CDH\ after adding some number of Cohen reals then
$X$ is meager in itself (hence a $\lambda$-set).}

\smallskip

Let $D$ be a countable dense subset of $X$ and let $\dot C$ be a name for a Cohen subset of $D$.
By genericity $\Vdash `` \dot C$ is a dense subset of $D$''. Let $\dot h$ be a $\mathbb C_\kappa$-name
for a homeomorphism of $X$ which sends $D$ onto $\dot C$.
Then there is a countable set of ordinals $J$ such that $\dot h$ is (equivalent to) a $\mathbb C_J$-name. Now, as $\mathbb C_J$ is countable,
$$
    X=\bigcup_{p \in \mathbb C_J} N_p,
$$
where
$$
    N_p=\{ x \in X : p \mbox{ decides }\dot h(x)\}.
$$
Now, it is easy to see that the set $N_p$ is closed and nowhere dense in $X$ (has only finite intersection with $D$), hence $X$ is meager in itself.
This completes the proof of the claim.

\smallskip

It follows that if $X$ is a \CDH-metric space of size less then $\mathfrak c$ then $X$ is a  $\lambda$-set, hence of size $\omega_1$ by Miller's result.
\end{proof}

Another corollary of Theorem \ref{main-thm} (or rather of its proof) is the following:

\begin{theorem}
Let $X\subseteq 2^\omega$ be an uncountable saturated $\lambda'$-set. Then for every
$D_0, D_1$ countable dense subsets of $X$ and $E_0, E_1$ countable dense subsets of $2^\omega\setminus X$
there is a homeomorphism $h$ of $2^\omega$ such that
\begin{enumerate}
\item $h[D_0]=D_1$,
\item $h[E_0]=E_1$, and
\item $h[X]=X$.
\end{enumerate}
\end{theorem}

This has the following interesting consequence\footnote{Recall that a space $X$ is \emph{completely Baire} if all of 
its non-empty closed subspaces are Baire.}:

\begin{corollary}
If $X\subseteq 2^\omega$ is an uncountable saturated $\lambda'$-set then $2^\omega\setminus X$ is a  (completely) Baire \CDH-space.
\end{corollary}

\begin{proof} The space $2^\omega\setminus X$ is clearly a \CDH-space. To see that it is completely Baire it suffices, by a result of 
Hurewicz \cite{hurewicz} (see \cite[pp. 78-79]{vanmill}), to show that $2^\omega\setminus X$ does not contain a closed copy of the rationals. Indeed, if $Q$ were such a copy
then $Q$ would be relatively $G_\delta$ both in $ 2^\omega\setminus X$ and in $Q\cup X$, hence it would be $G_\delta$ in $2^\omega$,
which is a contradiction.
\end{proof}

It should be noted that this result provides the first ZFC example of a metric Baire \CDH-space which is not completely metrizable. In \cite{hg-hr-CDH}
it is shown (extending a somewhat weaker result by Medini and Milovich \cite{medinimilovich}) that all non-meager P-ideals (seen as subspaces of $2^\omega$)
are {\CDH} and it is an older result of Marciszewski \cite{marciszewski} that  non-meager P-ideals are completely Baire. In fact, it has been recently
proved by Kunen, Medini and Zdomskyy \cite{kunen-medini-zdomskyy} that an ideal $\mathcal I$ is \CDH \  if and only if $\mathcal I$ is a non-meager P-ideal.
It is, however, an open problem whether non-meager P-ideals exist in ZFC (see \cite{nonmeagerp}). 

Unlike in the case of spaces which 
are meager in themselves, we do not know how to manipulate
cardinalities of Baire \CDH-spaces.

\begin{question}
 Is it consistent with ZFC to have a metric Baire \CDH-space without isolated points of size less than $\mathfrak c$?
\end{question}

The last comment on the construction deals with products. In a recent paper, Medini \cite{medini-CDH-products} constructed a zero-dimensional metric {\CDH} space whose
square is not \CDH \ under Martin's axiom, and  asked whether similar examples exist for higher dimensions. Our example answers one of his questions:

\begin{theorem}
Let $X\subseteq 2^\omega$ be an uncountable saturated $\lambda$-set. Then $X^n$ is {\CDH} for every $n\in\omega$ while $X^\omega$ is not \CDH.
\end{theorem}

\begin{proof}
First we note that $X^\omega$ is not \CDH . This follows from a recent result of Hern\'andez-Guti\'errez \cite{hg-arrow}, who showed that if  $X^\omega$ is {\CDH} and $X$ is crowded then
$X$ contains a copy of the Cantor set, which of course a $\lambda$-set  does not contain.

To finish the proof it suffices to show that $X^n$ is homeomorphic to a saturated  $\lambda$-set for every $n\in\omega$.
 The space $X^n$ is a $\lambda$-set by Lemma \ref{prodlamb}. Now consider the function $\psi_n: (2^\omega)^n\rightarrow 2^\omega$  defined by
$$\psi_n(x_0,\dots,x_{n-1})(n\cdot k +j) = x_j(k).$$
The function $\psi_n$ is a homeomorphism between  $(2^\omega)^n$ and $2^\omega$, and it should be obvious that if $X$ is a saturated subset of
$2^\omega$ then $\psi_n [X^n]$ is also saturated.
\end{proof}

\smallskip

The property of being saturated  describes how a set is situated inside of $2^\omega$. It would be interesting to characterize
meager in themselves \CDH-spaces internally. Recall that a zero-dimensional space $X$ is \emph{strongly homogeneous} if
all non-empty clopen subsets of $X$ are mutually homeomorphic. The methods of the proof of Theorem \ref{main-thm} can
be used to show the following:

\begin{proposition} Let $X$ be a $\lambda$-set such that 

\begin{enumerate}
\item $X$ is strongly homogeneous, and 
\item if $A$ and $B$ are subsets of $X$ with $A$ countable and $B$ nowhere dense, then there is a nowhere dense set $B'\subseteq X$ 
homeomorphic to $B$ with  $A\cap B'=\emptyset$.
\end{enumerate}
Then $X$ is \CDH.
\end{proposition}

Note that every \CDH \ space satisfies condition (2) while, as we have already seen, a meager in itself \CDH \ metric space is a $\lambda$-set.

\begin{question} Let $X$ be a zero-dimensional homogeneous meager in itself \CDH \ metric space. Is $X$ strongly homogeneous?
\end{question}

A similar question was raised by S.~V.~Medvedev in \cite[Question 1.]{medvedev}. A related question is the following:

\begin{question}
Is there a metric meager in itself  \CDH-space $X$ containg a countable dense set $D$ such that $X$ cannot be embedded in $X\setminus D$?
\end{question}

\section{Compact {\CDH} spaces from $\lambda'$-sets}

The literature on non-metrizable \CDH \ spaces is rather scarce. 
Arhangel'skii and van Mill in \cite{arh-vm-cdh-cardinality} showed that a \CDH{} space  has size at most $\cont$.
They also realized that no ZFC example of a compact non-metrizable \CDH{}  space was known. Consistent examples were known,
Stepr\={a}ns and Zhou \cite{step-zhou} observed that $2^\kappa$ is \CDH{} assuming $\kappa<\mathfrak p$.
A  non-metrizable hereditarily separable and hereditarily Lindel\"of compact \CDH{}  space was constructed 
in \cite{arh-vm-cdh-cardinality} assuming CH. In \cite{arh-vm-cdh-cardinality} it was also shown that assuming $2^\omega<2^{\omega_1}$, every  compact \CDH \ space is
first countable. 

A natural candidate for a compact non-metrizable \CDH{} space seemed to be Aleksandroff's  \emph{double arrow (or split interval)} space $\darrow$ \cite{alex-uryh-compact}.
It turned out that the space is not CDH \cite{arh-vm-cdh-cardinality}, in fact, it has $\mathfrak c$ many types of countable dense sets \cite{hg-arrow}.
It was conjectured that a slight modification of the double arrow should provide a ZFC example. The first attempts were proved unsuccessful by
Hern\'andez Guti\'errez \cite{hg-arrow}, who showed that both $\darrow\times 2^\omega$ and $\darrow^\omega$ are not \CDH. These results follow directly from the following:

\begin{proposition}{\cite{hg-arrow}} Let $X$ and $Y$ be crowded spaces with countable $ \pi$-bases. If the product $X\times Y$ is \CDH , then either both $X$ and $Y$ contain a copy of $2^\omega$ or neither does.
\end{proposition}

Here we will show that  the \emph{double arrow space over a saturated $\lambda'$-set} is \CDH, hence there are compact non-metrizable \CDH{} spaces in ZFC, indeed.

\smallskip

First let us fix some notation for linearly ordered spaces. Given a linearly ordered topological space $(X,<)$, a function $f:X\to X$ and $x\in X$ which is 
neither the least nor the largest element of $X$,
we will say that $f$ is \emph{monotone at} $x$ if there exist $a,b\in X$ with $a<x<b$ such that either $f[(a,x)]\subset(\leftarrow,f(x))$ and
$f[(x,b)]\subset(f(x),\rightarrow)$, or $f[(a,x)]\subset(f(x),\rightarrow)$ and $f[(x,b)]\subset(\leftarrow,f(x))$.

We consider here  $\cantorset$ as a linearly ordered topological space ordered lexicographically (or, equivalently) consider it as a subspace of $\reals$
with the induced order. Let $Q_i=\{x\in\cantorset:\exists n<\omega\ \forall m\geq n\ (x(m)=i)\}$ for $i\in 2$ and $Q=Q_0\cup Q_1$. The set $Q_0$ consists of the
least element of $\cantorset$ and all those points that have an immediate predecessor. Similarly, $Q_1$ consists of the greatest element of $\cantorset$ and
all those points that have an immediate successor.

\begin{definition}
 For each $X\subset\cantorset\setminus Q$, let $\darrow(X)$ be the set $((\cantorset\setminus Q_0)\times\{0\})\cup((Q_0\cup X)\times\{1\})$ with the order
topology given by the lexicographic order.
\end{definition}

It is easy to see that $\darrow(X)$ is a $0$-dimensional separable compact Hausdorff space of weight $|X|$.
Note that $\darrow(C)$ is homeomorphic to the Cantor set whenever $C\subset\cantorset\setminus Q$ is countable. In fact, using a 
result of Ostaszewski's \cite{arrows} it can be easily shown that every compact $0$-dimensional separable linearly ordered space without isolated points is 
homeomorphic to $\darrow(X)$ for some
$X\subset\cantorset\setminus Q$.

If $X\subset Y\subset\cantorset\setminus Q$, let $\pi_{X}\sp{Y}:\darrow(Y)\to\darrow(X)$ be the natural projection defined by
$$
\pi_{X}\sp{Y}(\pair{x,t})=\left\{
\begin{array}{rl}
\pair{x,0}&\textrm{ if }x\in Y\setminus X,\\
\pair{x,t}&\textrm{ otherwise.}
\end{array}
\right.
$$

 The following is not hard to see.

\begin{lemma}\label{lifting}
Let $X\subset Y\subset\cantorset\setminus Q$ and let $h:\darrow(X)\to \darrow(X)$ be a homeomorphism.  Assume that $h$ is monotone at each point of the form $\pair{x,0}$ with $x\in Y\setminus X$ and $h[(Y\setminus X)\times\{0\}]=(Y\setminus X)\times\{0\}$. Then there exists a homeomorphism $H:\darrow(Y)\to\darrow(Y)$ such that $\pi_{X}\sp{Y}\circ H=h\circ\pi_{X}\sp{Y}$.
\end{lemma}

\begin{proposition}\label{homeoinCantor}
Let $D_0$ and $D_1$ be countable dense subsets of $\cantorset$ and  let $W\subset\cantorset$ be such that $W\cap(D_0\cup D_1\cup Q)=\emptyset$ and $W\cup D_0\cup D_1$ is a $\lambda$-set. Furthermore, assume that the following condition holds.
\begin{quote}$(\ast)$
For every two non-empty clopen intervals $I,J$ of $\cantorset$, there is an order isomorphism $f:I\to J$ such that $f[I\cap W]=J\cap W$.
\end{quote}
Then there is a homeomorphism $h:\cantorset\to\cantorset$ such that $h[D_0]=D_1$, $h[W]=W$ and $h$ is monotone at each point of $W$.
\end{proposition}
\begin{proof}
Start by fixing some metric for $\cantorset$. Since $W\cup D_0\cup D_1$ is a $\lambda$-set, there exists a collection $\{K_n:n<\omega\}$ of closed nowhere dense subsets of $\cantorset$ such
that $W\subset\bigcup\{K_n:n<\omega\}$ and $(D_0\cup D_1)\cap K_n=\emptyset$ for each $n<\omega$. Let $D_0=\{d_n\sp0:n<\omega\}$, $D_1=\{d_n\sp1:n<\omega\}$.

We will contruct the homeomorphism we are looking for by defining a sequence of autohomeomorphisms of $\cantorset$. In step $n<\omega$, we will construct a
homeomorphism $h_n:\cantorset\to \cantorset$ and a pair of closed nowhere dense sets $G_n\sp{0}$ and $G_n\sp{1}$ that have the following properties.
\begin{itemize}
\item[(a)] For each $i\in 2$, $G_n\sp{i}\cup K_n\subset G_{n+1}\sp{i}$.
\item[(b)] For each $i\in 2$, $D_i\cap G_n\sp{i}$ is finite and contains $\{d_k\sp{i}:k< n\}$.
\item[(c)] $h_n[G_n\sp0]=G_n\sp1$, $h_n[D_0\cap G_n\sp0]=D_1\cap G_n\sp1$ and $h_n[W]=W$.
\item[(d)] For each $k<n$, $h_n\!\!\restriction_{G_k\sp0}=h_k$.
\end{itemize}

We will also need a KR-cover $\langle \I_n\sp{0},\I_n\sp{1},\alpha_n\rangle$ for $\langle \cantorset\setminus G_n\sp{0},\cantorset\setminus G_n\sp{1},h_n\!\!\restriction_{G_n\sp{0}}\rangle$. In order  to ensure the monotonicity of $h$ at all points of $W$, we will construct these KR-covers with the following additional properties.

\begin{itemize}
\item[(e)] Given $i\in2$, if $I\in\I_n\sp{i}$, then $I$ is a clopen interval of the form $[q_0,q_1]$ where $q_j\in Q_j$ for $j\in2$ and $I$ has diameter $<\frac{1}{n+1}$.
\item[(f)] For all $i\in2$, $\I_{n+1}\sp{i}$ refines $\I_n\sp{i}$; that is, every element of $\I_{n+1}\sp{i}$ is contained in an element of $\I_n\sp{i}$.
\item[(g)] If $I\in\I_n\sp{0}$, then $h_n[I]=\alpha_n(I)$, $h_n[I\cap W]=\alpha_n(I)\cap W$ and $h_n\!\!\restriction_{I}$ is an order isomorphism.
\item[(h)] Given $m<n$, if $I\in\I_{m}\sp{0}$ and $J\in\I_{n}\sp{0}$ then $J\subset I$ if and only if $\alpha_n(J)\subset\alpha_m(I)$.
\item[(i)] Given $x\in W$, let $k=\min\{m<\omega:x\in G_m\sp{0}\}$. Then there are $a,b\in Q$ with $a<x<b$ such that if $n\geq k$, $I,J\in\I_n\sp{0}$, $I\subset[a,x]$ and $J\subset[x,b]$ then $\alpha_n(I)\subset(\leftarrow,h_m(x)]$ and $\alpha_n(J)\subset[h_m(x),\rightarrow)$.
\end{itemize}

To start with the construction, let $G_0\sp0=G_0\sp1=\emptyset$, $\I_0\sp0=\I_0\sp1=\{\cantorset\}$, $\alpha_0=\{\pair{\cantorset,\cantorset}\}$ and let $h_0$ be the identity function. 

In step $n+1$ of the construction, first notice that the homeomorphism $h_n:\cantorset\to \cantorset$ fixes $W$ and is monotone at every point of $W$ by (g) and (i). Now, let $G=G_n\sp0\cup K_n\cup\{d_n\sp0\}$. The point $d\sp{0}_n$ may or may not be already contained in $G_n$. We will describe the construction in the case that $d_n\sp0\notin G_n$ and $d\sp{0}_n$ does not have an immediate predecessors or successors, the other cases can be treated in an analogous way. Let $I_0\in\I_n\sp{0}$ be such that $d_n\sp0\in I_0$.

If $J\in\I_n\sp{0}\setminus\{I_0\}$ then partition $J$ into finitely many clopen intervals with endpoints contained in $Q\setminus G$. This is indeed possible because $Q$ is dense in the dense open subset $J\setminus G$ of $J$ and the endpoints of $J$ are already in $Q$.  Call this collection $\I(J,0)$. Let $\I(J,1)=\{h_n[K]:K\in\I(J,0)\}$ and let $\alpha_J:\I(J,0)\to\I(J,1)$ be such that $\alpha_J(K)=h_n[K]$ for all $K\in\I(J,0)$. Notice that by (g), $\I(J,1)$ will also consist of clopen intervals.

Since $G_n\sp{0}\cup K_n$ intersects $I_0$ in a closed subset that does not contain $d\sp{0}_n$, and $d\sp{0}_n$ does not have immediate predecessors or successors, there are $a_j\in Q_j\cap I_0$ for $j\in 2$ such that $d\sp{0}_n\in(a_0,a_1)$ and $[a_0,a_1]\cap (G_n\sp0\cup K_n)=\emptyset$. Let $b_j=h_n(a_j)$ for $j\in 2$, notice that $b_j\in Q_j$ for $j\in 2$ because $h_n$ is an order isomorphism. Notice that this guarantees that $[a_0,a_1]$ and $[b_0,b_1]$ are clopen intervals. Choose an $e\in(b_0,b_1)\cap D_1$ that has no immediate predecessors or successors.

Partition $[a_0,a_1]\setminus\{d\sp{0}_n\}=\bigcup\{U\sp{i}_m:m<\omega,i\in 2\}$ and $[b_0,b_1]\setminus\{e\}=\bigcup\{V\sp{i}_m:m<\omega,i\in 2\}$ so that

\begin{itemize}
\item[(1)] if $k<\omega$ and $i\in 2$, $U\sp{i}_m$ and $V\sp{i}_m$ are non-empty clopen intervals of diameter $<\frac{1}{n+2}$,
\item[(2)] if $k<\omega$, $p\in U\sp0_m$, $q\in U\sp0_{m+1}$, $r\in U\sp1_m$ and $s\in U\sp1_{m+1}$, then $p<q<d\sp{0}_n<s<r$, and
\item[(3)] if $k<\omega$, $p\in V\sp0_m$, $q\in V\sp0_{m+1}$, $r\in V\sp1_m$ and $s\in V\sp1_{m+1}$, then $p<q<e<s<r$.
\end{itemize}
Let
\begin{eqnarray}
\I(I_0,0,0)&=\{U\sp{i}_m:m<\omega,i\in 2\}&\textrm{ and}\nonumber\\
\I(I_0,1,0)&=\{V\sp{i}_m:m<\omega,i\in 2\}.&\nonumber
\end{eqnarray}

Also, the set $I_0\setminus[a_0,a_1]$ can be partitioned into finitely many clopen intervals with diameter $<\frac{1}{n+2}$, call this collection $\I(I_0,0,1)$. 
Let $\I(I_0,1,1)=\{h_n[K]:K\in\I(I_0,0,1)\}$ and let $\I(I_0,i)=\I(I_0,i,0)\cup\I(I_0,i,1)$ for $i\in 2$.

Next we define a bijection $\beta:\I(I_0,0)\to\I(I_0,1)$ as follows: For $m<\omega$ and $i\in 2$ let  $\beta(U\sp{i}_m)=V\sp{i}_m$, and for $K\in \I(I_0,0,1)$ let $\beta(K)=h_n[K]$.

Put
\begin{eqnarray}
\I\sp{j}&=&\bigcup\{\I(J,j):J\in\I_n\sp{j}\},\textrm{ for }j\in 2\textrm{ and}\nonumber\\
\alpha&=&\beta\cup\Big(\bigcup\{\alpha_J:J\in\I_n\sp{0}\setminus\{I_0\}\}\Big),\nonumber
\end{eqnarray}
so $\alpha:\I\sp{0}\to\I\sp{1}$ is a bijection. We also define a homeomorphism $H:\cantorset\to \cantorset$. For each $m<\omega$ and $i\in 2$ we use property $(\ast)$ to find an order isomorphism $f\sp{i}_m:U\sp{i}_m\to V\sp{i}_m$. Then we let
$$
H(x)=\left\{
\begin{array}{ll}
h_n(x), & \textrm{ if } x\in \cantorset\setminus[a_0,a_1],\\
e, & \textrm{ if }x=d\sp0_n,\\
f\sp{i}_m(x), & \textrm{ if } x\in U\sp{i}_m, \textrm{ for some }m<\omega, i\in2.
\end{array}
\right.
$$

Then $H:\cantorset\to\cantorset$ is a homeomorphism. Notice that $(\I\sp0,\I\sp1,\alpha)$ is a KR-cover for $(\cantorset\setminus G,\cantorset\setminus H[G],H\!\!\restriction_{G})$. Moreover, it is not hard to see that $G$, $\alpha$, $\I\sp0$, $\I\sp1$ and $H$ have the desired properties from the list (a) to (i) when $i=0$, but perhaps these properties do not hold for $i=1$. In order to finish step $n+1$, we have to do another refinement to both the homeomorphism and the covers so that (a) to (i) hold when $i=1$ too. This is entirely analogous to what we have done so we leave the details to the reader.

Finally, we will show that there is a homeomorphism $h:\cantorset\to \cantorset$ that extends $h_n\!\!\restriction_{G_n\sp{0}}$ for every $n<\omega$.  
Assume that $x\notin\bigcup\{G_n\sp0:n<\omega\}$. For every $n<\omega$, let $I_n^x$ be the unique element of $\I_n\sp0$ that contains $x$. 
It is easy to check that $\bigcap\{\alpha_n(I_n^x):n<\omega\}=\{y\}$ for some $y\in\cantorset$. Let $h(x)=y$. 
We leave the verification that this indeed  defines a homeomorphism to the reader. 
In this way, by properties (a), (b) and (c), we immediately obtain that $h[W]=W$ and $h[D_0]=D_1$.
\end{proof}

\begin{theorem}\label{compact} Let $Y\subset\cantorset\setminus Q$ be a saturated $\lambda\sp\prime$-set. 
Then $X=\darrow(Y)$ is a compact linearly ordered \CDH{} space of weight $|Y|$.

\end{theorem}
\begin{proof} To show that $X$ is CDH, let $D$ and $E$ be countable dense subsets of $\darrow(Y)$.

Define $G=\{x\in Y:\pair{x,i}\in D\cup E\textrm{ for some }i\in 2\}$, $F=G\sp\ast$ and $Z=Y\setminus F$. Notice that $Z$ is a saturated
$\lambda\sp\prime$-set in $\cantorset$ and $Z\cap Q=\emptyset$. Consider the space $\darrow(F)$. Note that
$\pi_{F}\sp{Z}:X\to\darrow(F)$ is the identity when restricted to $D\cup E$. Let $D\sp\prime=\pi_{F}\sp{Z}[D]$ and $E\sp\prime=\pi_{F}\sp{Z}[E]$,
these are countable dense subsets of $\darrow(F)$.

We would like to find a homeomorphism of $\darrow(F)$ that can be lifted to a homeomorphism of $X$ using Lemma \ref{lifting}. Thus, let us argue that we can use Proposition \ref{homeoinCantor} to find an appropriate autohomeomorphism of the Cantor set $\darrow(F)$. Let $W=Z\times2$.

Clearly, $Z$ is a $\lambda\sp\prime$-set in $\cantorset$ and from this it is easy to prove that $W$ is a $\lambda\sp\prime$-set in $\darrow(F)$. Notice that the set $(F\times2)\cup (Q_0\times\{1\})\cup(Q_1\times\{0\})$ is the set of all the points of $\darrow(F)$ with no immediate succesor or no immediate predecessor along with the least and greatest elements of $\darrow(F)$, hence $W$ does not contain any of these points. Also, notice that $D\sp\prime$ and $E\sp\prime$ do not intersect $W$.

Then we are left to prove that the condition $(\ast)$ in Proposition \ref{homeoinCantor} holds. Let $I$ and $J$ be two clopen intervals of
$\darrow(F)$. Then there are $a,b,c,d\in\cantorset$ with $a<b$, $c<d$, $I=[\pair{a,1},\pair{b,0}]$ and $J=[\pair{c,1},\pair{d,0}]$. Notice that $a$ and $c$ have no immediate successor, and $b$ and $d$ have no immediate predecessor. Therefore, there are 
$\{s_n:n\in\Z\}\cup\{t_n:n\in\Z\}\subset\cantortree$ such that
\begin{itemize}
\item[(i)] $(a,b)=\bigcup\{[s_n]:n\in\Z\}$,
\item[(ii)] $(c,d)=\bigcup\{[t_n]:n\in\Z\}$,
\item[(iii)] if $n,m\in\Z$, $n<m$, $x\in[s_n]$ and $y\in[s_m]$, then $x<y$, and
\item[(iv)] if $n,m\in\Z$, $n<m$, $x\in[t_n]$ and $y\in[t_m]$, then $x<y$.
\end{itemize}
Define $g:[a,b]\to[c,d]$ by
$$
g(x)=\left\{
\begin{array}{ll}
c, & \textrm{ if }x=a,\\
d, & \textrm{ if }x=b,\\
h_{s_n,t_n}(x), & \textrm{ if }n\in\Z\textrm{ and }x\in[s_n].
\end{array}
\right.
$$
for all $x\in[a,b]$. Then $g$ is an order isomorphism and, since $F$ and $Z$ are saturated, $g[F\cap[a,b]]=F\cap[c,d]$ and $g[Z\cap[a,b]]=Z\cap[c,d]$.
Define $f:I\to J$ as $f(\pair{x,t})=\pair{g(x),t}$ for all $\pair{x,t}\in I$, this function is well-defined, and is, in fact, an order isomorphism such that
$f[I\cap W]=J\cap W$.

Thus, using Proposition \ref{homeoinCantor}, there is a homeomorphism $h:\darrow(F)\to\darrow(F)$ such that $h[D\sp\prime]=E\sp\prime$ and $h$ is monotone
at each point of $W$. By Lemma \ref{lifting}, there is a homeomorphism $H:X\to X$ such that $\pi_{F}\sp{Z}\circ H=h\circ \pi_{F}\sp{Z}$, so $H[D]=E$. 
\end{proof}

Recall that every set of size less than $\mathfrak b$ is a  $\lambda'$-set. So,

\begin{corollary} There exists a linearly ordered, compact, $0$-dimensional CDH space of weight $\kappa$ for any $\kappa$ of 
size less than $\mathfrak b$.
\end{corollary}

More importantly, there is a $\lambda'$-set of size $\aleph_1$ (see \cite{Miller84}). Hence,

\begin{theorem} There exists a linearly ordered, compact, $0$-dimensional CDH space of weight $\omega_1$.
\end{theorem}

The use of the $\lambda'$-sets here is necessary.
If $X\subset\cantorset\setminus Q$ is a Baire space, by using the same method as in \cite{arh-vm-cdh-cardinality} it is possible to prove that the space
$\darrow(X)$ is not {\CDH}. Also, since in the Cohen model all $\lambda$-sets have cardinality $\omega_1$ (\cite[Theorem 22]{Miller93}) we do not have examples
of compact CDH spaces of weight $\cont$ in ZFC.

\begin{question}
Is there a compact  {\CDH} space of weight $\cont$ in ZFC?
\end{question}

\begin{question}
Is there a non-metrizable {\CDH} continuum?
\end{question}

\bigskip

{\bf Acknowledgments.} The authors wish to thank S.~V.~Medvedev and the anonymous referee for commenting on the paper. They
pointed out several inaccuracies that appeared in an earlier version of the paper and helped us to improve the paper in general.

\def\cprime{$'$}
\makeatletter \renewcommand{\@biblabel}[1]{\hfill[#1]}\makeatother


\begin{thebibliography}{10}

\bibitem{alex-uryh-compact}
P. Alexandroff, P. Uryshon, {\em M\'emoire sur les espaces topologiques compacts}, Verh. Akad. Wetensch. Amsterdam 14 (1929).

\bibitem{arh-vm-cdh-cardinality}
A. V. Arhangel'ski\u \i, J. van Mill, {\em On the cardinality of countable dense homogeneous spaces}, Proc. Amer. Math. Soc. 141 (2013), 4031--4038.

\bibitem{arh-vm-homogeneity-review}
A. V. Arhangel'ski\u \i, J. van Mill, {\em Topological Homogeneity}, in Recent Progress in General Topology, III, 2013, p. 1--68, Atlantis Press.

\bibitem{CDH_MA}
S. Baldwin, R. E. Beaudoin, {\em Countable dense homogeneous spaces under Martin's axiom}, Israel J. Math. 65 (1989), no. 2, 153--164.

\bibitem{bennett}
R. Bennett, {\em Countable dense homogeneous spaces}, Fund. Math. 74 (1972), no. 3, 189--194.

\bibitem{bp:estimated}
C.~Bessaga and A.~Pe{\l}czy\'{n}ski, {\em The estimated extension theorem
  homogeneous collections and skeletons, and their application to the
  topological classification of linear metric spaces and convex sets}, Fund.
  Math. 69 (1970), 153--190.

\bibitem{fons86}
A.~J.~M. van Engelen, {\em {Homogeneous zero-dimensional absolute Borel sets}},
  CWI Tract, vol.~27, Centre for Mathematics and Computer Science, Amsterdam,
  1986.

\bibitem{engelking:dim:nieuw}
R.~Engelking, {\em Theory of dimensions finite and infinite}, Heldermann
  Verlag, Lemgo, 1995.

\bibitem{FarahHrusakMartinez05}
I.~Farah, M.~Hru{\v{s}}{\'a}k, and C.~Mart\'{\i}nez Ranero, {\em A countable
  dense homogeneous set of reals of size $\aleph_1$}, Fund. Math. 186
  (2005), 71--77.

\bibitem{fitz-CDH-locconn}
B. Fitzpatrick, Jr., {\em A note on countable dense homogeneity}, Fund. Math. 75 (1972), no. 1, 33--34.

\bibitem{openprobCDH}
B. Fitzpatrick, Jr. and H-X. Zhou, {\em Some Open Problems in Densely Homogeneous Spaces}, in Open Problems in Topology (ed. J. van Mill and M. Reed), 1984, pp. 251--259, North-Holland, Amsterdam.

\bibitem{FitzpatrickZhou92}
B.~Fitzpatrick, Jr. and H-X. Zhou, {\em Countable dense homogeneity and the
  Baire property}, Top. Appl. 43 (1992), 1--14.

\bibitem{hg-arrow}
R. Hern\'andez-Guti\'errez, {\em Countable dense homogeneity and the double arrow space}, Top. Applications 160 (2013), no. 10, 1123--1128.

\bibitem{hg-hr-CDH}
R. Hern\'andez-Guti\'errez, M. Hru\v s\'ak, {\em Non-meager $P$-filters are countable dense homogeneous}, Colloq. Math. 130 (2013), 281--289.

\bibitem{HrusakZamora05}
M.~Hru{\v{s}}{\'a}k and B.~Zamora Avil\'es, {\em Countable dense homogeneity of
  definable spaces}, Proc. Amer. Math. Soc.   133 (2005), 3429--3435.

\bibitem{hurewicz}
W.~Hurewicz, {\em Relativ perfekte Teile von Punktmengen und Mengen (A)}, Fund. Math. 12 (1928),
78--109.

\bibitem{nonmeagerp} W. Just, A. R. D. Mathias, K. Prikry, P. Simon, {\em On the existence of large p-ideals}, J. Symbolic Logic 55 (1990), no. 2, 457--465.

\bibitem{KnasterReichbach}
B.~Knaster and M.~Reichbach, {\em {Notion d'homog\'en\'eit\'e et prolongements
  des hom\'eomorphies}}, Fund. Math.  40 (1953), 180--193.

\bibitem{kunen-medini-zdomskyy}
K.~Kunen, A.~Medini and L.~Zdomskyy, {\em  Seven characterizations of non-meager $P$-filters}, pre-print (2013). 

\bibitem{kuratowski-lambda}
K. Kuratowski, {\em Sur une famille d'ensembles singuliers}, Fund. Math. 21  (1933), no. 1, 127--128.

\bibitem{lusin-lambda}
N. Luzin, {\em Sur l'existence d'un ensemble nond\'enombrable qui est de premi\`ere categorie sur tout ensemble parfait}, Fund. Math.  2, 155--157.

\bibitem{marciszewski} 
W. Marciszewski, {\em P-filters and hereditary Baire function spaces}, Topology Appl. 89 (1998), no. 3, 241--247.

\bibitem{medvedev} S.~V.~Medvedev, {\em About closed subsets of spaces of first category}, Topology Appl. 159 (2012), no. 8, 2187--2192.

\bibitem{vm:38}
J.~van Mill, {\em Characterization of some zero-dimensional separable metric
  spaces}, Trans. Amer. Math. Soc.  264 (1981), 205--215.

\bibitem{vanmill}
J.~van Mill, {\em  The infinite-dimensional topology of function spaces}, North-Holland Mathematical
Library, 64, North-Holland Publishing Co., Amsterdam, 2001.

\bibitem{Miller93}
A.~W. Miller, {\em Special sets of reals}, Set theory of the reals ({R}amat
  {G}an, 1991), Israel Math. Conf. Proc., vol.~6, Bar-Ilan Univ., Ramat Gan,
  1993, pp.~415--431.

\bibitem{Miller84}
A. W. Miller, {\em Special subsets of the real line}, Handbook of set-theoretic topology, 201--233, North-Holland, Amsterdam, 1984.

\bibitem{medini-CDH-products}
A. Medini, {\em Products and Countable Dense Homogeneity}, preprint

\bibitem{medinimilovich}
A. Medini and D. Milovich, {\em The topology of ultrafilters as subspaces of $2\sp\omega$}, Topology Appl. 159, (2012), no. 5, 1318--1333.

\bibitem{arrows}
A. J. Ostaszewski, {\em A characterization of compact, separable, ordered spaces}, J. London Math. Soc. (2) 7 (1974), 758--760.

\bibitem{rothberger-lambda}
F. Rothberger, {\em Sur un ensemble de premi\`re cat\'egorie qui est d\'epourvou de la propri\'et\'e $\lambda$}, Fund. Math.  32, (1939) 294--300.

\bibitem{saltsman1}
W. L. Saltsman, {\em Concerning the existence of a connected, countable dense homogeneous subset of the plane which is not strongly locally homogeneous}, Topology Proc. 16 (1991), 137--176.

\bibitem{saltsman2}
W. L. Saltsman, {\em Concerning the existence of a nondegenerate connected, countable dense homogeneous subset of the plane which has a rigid open subset}, Topology Proc. 16 (1991), 177--183.

\bibitem{sierpinski-lambda}
W. Sierpi\'nski, {\em Sur une propri\'et\'e additive d'ensembles}, Computes Rendus Soc. Sci. Varsovie,  30, 257--259.

\bibitem{silver-equiv-classes}
J. H. Silver, {\em Counting the number of equivalence classes of Borel and coanalytic equivalence relations}, Ann. Math. Logic 18 (1980), no. 1, 1--28.

\bibitem{step-zhou}
J. Stepr\={a}ns and H. X. Zhou, {\em  Some results on CDH spaces, I},  Top. Appl.  28 (1988) 147--154.

\end{thebibliography}
\end{document}